\documentclass[a4paper]{amsart}
\usepackage[english]{babel}
\usepackage{amsmath,amsthm,amssymb,amsfonts}

\newtheorem{thm}{Theorem}
\newtheorem{prp}[thm]{Proposition}

\newtheorem{lem}[thm]{Lemma}
\newtheorem*{assA}{Assumption \AssA}
\newtheorem*{assB}{Assumption \AssB{t}}

\newcommand{\AssA}{$(\mathrm{A})$}
\newcommand{\AssB}[1]{$(\mathrm{B}_{#1})$}

\DeclareMathOperator{\Kern}{\mathcal{K}}
\DeclareMathOperator{\supp}{\mathrm{supp}}
\DeclareMathOperator{\tr}{\mathrm{tr}}

\DeclareMathOperator{\Span}{\mathrm{span}}

\newcommand{\Heis}{\mathrm{H}}

\newcommand{\N}{\mathbb{N}}
\newcommand{\Z}{\mathbb{Z}}
\newcommand{\R}{\mathbb{R}}
\newcommand{\C}{\mathbb{C}}

\newcommand{\tc}{\,:\,}

\newcommand{\leftopenint}{\left]}
\newcommand{\rightopenint}{\right[}
\newcommand{\leftclosedint}{\left[}
\newcommand{\rightclosedint}{\right]}

\newcommand{\vecL}{\mathbf{L}}
\newcommand{\vecU}{\mathbf{U}}

\newcommand{\Ell}{\mathcal{L}}

\newcommand{\id}{\mathrm{id}}

\newcommand{\lie}{\mathfrak}
\newcommand{\fst}{\lie{v}}
\newcommand{\ctr}{\lie{z}}
\newcommand{\dctr}{\lie{\dot z}}

\newcommand{\done}{{d_1}}
\newcommand{\dtwo}{{d_2}}
\newcommand{\tdone}{{\tilde\done}}
\newcommand{\jone}{j}
\newcommand{\jtwo}{k}

\newcommand{\defeq}{\mathrel{:=}}

\begin{document}
\title[Spectral multipliers on Heisenberg-Reiter groups]{Spectral multipliers on Heisenberg-Reiter and related groups}
\author{Alessio Martini}
\address{Alessio Martini \\ Mathematisches Seminar \\ Christian-Albrechts-Universit\"at zu Kiel \\ Ludewig-Meyn-Str.\ 4 \\ D-24118 Kiel \\ Germany}
\email{martini@math.uni-kiel.de}
\thanks{The author gratefully acknowledges the support of the Alexander von Humboldt Foundation.}

\begin{abstract}
Let $L$ be a homogeneous sublaplacian on a 2-step stratified Lie group $G$ of topological dimension $d$ and homogeneous dimension $Q$. By a theorem due to Christ and to Mauceri and Meda, an operator of the form $F(L)$ is bounded on $L^p$ for $1 < p < \infty$ if $F$ satisfies a scale-invariant smoothness condition of order $s > Q/2$. Under suitable assumptions on $G$ and $L$, here we show that a smoothness condition of order $s > d/2$ is sufficient. This extends to a larger class of $2$-step groups the results for the Heisenberg and related groups by M\"uller and Stein and by Hebisch, and for the free group $N_{3,2}$ by M\"uller and the author.
\end{abstract}

\maketitle

\section{Introduction}

Let $L$ be a homogeneous sublaplacian on a stratified Lie group $G$ of homogeneous dimension $Q$. Since $L$ is a positive selfadjoint operator on $L^2(G)$, a functional calculus for $L$ is defined via the spectral theorem and, for all Borel functions $F : \R \to \C$, the operator $F(L)$ is bounded on $L^2(G)$ whenever the ``spectral multiplier'' $F$ is bounded. As for the $L^p$-boundedness for $p \neq 2$ of $F(L)$, a sufficient condition in terms of smoothness properties of the multiplier $F$ is given by a theorem of Mihlin-H\"ormander type due to Christ \cite{christ_multipliers_1991} and Mauceri and Meda \cite{mauceri_vectorvalued_1990}: the operator $F(L)$ is of weak type $(1,1)$ and bounded on $L^p(G)$ for all $p \in \leftopenint 1,\infty \rightopenint$ whenever
\[
\|F\|_{MW_2^s} \defeq 
\sup_{t > 0} \| F(t \cdot) \, \eta \|_{W_2^s} < \infty
\]
for some $s > Q/2$, where $W_2^s(\R)$ is the $L^2$ Sobolev space of fractional order $s$ and $\eta \in C^\infty_c(\leftopenint 0,\infty \rightopenint)$ is a nontrivial auxiliary function.

A natural question that arises is if the smoothness condition $s > Q/2$ is sharp. This is clearly true when $G$ is abelian, so $Q$ coincides with the topological dimension $d$ of $G$, and $L$ is essentially the Laplace operator on $\R^d$. Take however the smallest nonabelian example of a stratified group, that is, the Heisenberg group $\Heis_1$, which is defined by endowing $\R \times \R \times \R$ with the group law
\begin{equation}\label{eq:heisenberggrouplaw}
(x,y,u) \cdot (x',y',u') = (x+x',y+y',u+u'+(xy'-x'y)/2)
\end{equation}
and with the automorphic dilations
\begin{equation}\label{eq:dilations}
\delta_t(x,y,u) = (tx,ty,t^2 u).
\end{equation}
$\Heis_1$ is a $2$-step stratified group, and the homogeneous dimension of $\Heis_1$ is $4$. Nevertheless, a result by M\"uller and Stein \cite{mller_spectral_1994} and Hebisch \cite{hebisch_multiplier_1993} shows that, for a homogeneous sublaplacian on $\Heis_1$, the smoothness condition on the multiplier can be pushed down to $s > d/2$, where $d = 3$ is the topological dimension of $\Heis_1$ (in \cite{mller_spectral_1994} it is also proved that the condition $s > d/2$ is sharp). Such an improvement of the Christ-Mauceri-Meda theorem holds not only for $\Heis_1$, but for the larger class of M\'etivier groups (and for direct products of M\'etivier and abelian groups), and also for differential operators other than sublaplacians (see, e.g., \cite{hebisch_multiplier_1995,martini_joint_2012}). However it is still an open question whether, for a general stratified Lie group (or even for a general $2$-step stratified group), the homogeneous dimension in the smoothness condition can be replaced by the topological dimension.

The aim of this paper is to extend the class of the $2$-step stratified groups and sublaplacians for which the smoothness condition in the multiplier theorem can be pushed down to half the topological dimension.

Take for instance the Heisenberg-Reiter group $\Heis_{\done,\dtwo}$ (cf.\ \cite{torreslopera_cohomology_1985}), defined by endowing $\R^{\dtwo \times \done} \times \R^\done \times \R^\dtwo$ with the group law \eqref{eq:heisenberggrouplaw} and the automorphic dilations \eqref{eq:dilations}; here however $\R^{\dtwo \times \done}$ is the set of the real $\dtwo \times \done$ matrices, and the products $xy',x'y$ in \eqref{eq:heisenberggrouplaw} are interpreted in the sense of matrix multiplication. $\Heis_{\done,\dtwo}$ is a $2$-step stratified group of homogeneous dimension $Q = \done\dtwo + \done + 2\dtwo$ and topological dimension $d = \done\dtwo + \done + \dtwo$. Despite the formal similarity with $\Heis_1$, the group $\Heis_{\done,\dtwo}$ does not fall into the class of M\'etivier groups, unless $\dtwo = 1$ (in fact, $\Heis_{\done,1}$ is the $(2\done+1)$-dimensional Heisenberg group $\Heis_\done$). Nevertheless, the technique presented here allows one to handle the case $\dtwo > 1$ too.

Namely, let $X_{1,1},\dots,X_{\dtwo,\done},Y_1,\dots,Y_\done,U_1,\dots,U_\dtwo$ be the left-invariant vector fields on $\Heis_{\done,\dtwo}$ extending the standard basis of $\R^{\dtwo \times \done} \times \R^\done \times \R^\dtwo$ at the identity, and define the homogeneous sublaplacian $L$ by
\[
L = -\sum_{\jone=1}^\done \sum_{\jtwo=1}^\dtwo X_{\jtwo,\jone}^2 - \sum_{\jone=1}^\done Y_\jone^2.
\]
Then a particular instance of our main result reads as follows.

\begin{thm}\label{thm:mhheisreit}
Suppose that a function $F : \R \to \C$ satisfies
\[\|F\|_{MW^s_2} < \infty\]
for some $s > d/2$. Then the operator $F(L)$ is of weak type $(1,1)$ and bounded on $L^p(\Heis_{\done,\dtwo})$ for all $p \in \leftopenint 1,\infty\rightopenint$.
\end{thm}

To the best of our knowledge, this result is new, at least in the case $\dtwo > \done$. In fact, in the case $\dtwo \leq \done$, the extension described in \cite{martini_joint_2012} of the technique of \cite{hebisch_multiplier_1993,hebisch_multiplier_1995} would give the same result. However the technique presented here is different, and yields the result irrespective of the parameters $\done,\dtwo$.

The left quotient of $\Heis_{\done,\dtwo}$ by the subgroup $\R^{\dtwo \times \done} \times \{0\} \times \{0\}$ gives a homogeneous space diffeomorphic to $\R^\done \times \R^\dtwo$, and the sublaplacian $L$ corresponds in the quotient to a Grushin operator. In recent joint works with Adam Sikora \cite{martini_grushin} and Detlef M\"uller \cite{martini_grushin2}, we proved for these Grushin operators on $\R^\done \times \R^\dtwo$ a sharp spectral multiplier theorem of Mihlin-H\"ormander type, where the smoothness requirement is again half the topological dimension of the ambient space.

The proofs in \cite{martini_grushin,martini_grushin2} rely heavily on properties of the eigenfunction expansions for the Hermite operators. Since a homogeneous sublaplacian on a $2$-step stratified group reduces to a Hermite operator in almost all irreducible unitary representations of the group, it is conceivable that an adaptation of the methods of \cite{martini_grushin,martini_grushin2} may give an improvement to the multiplier theorem for $2$-step stratified groups, even outside of the M\'etivier setting. A first result in this direction is shown in \cite{martini_n32}, where the free $2$-step nilpotent Lie group $N_{3,2}$ on three generators is considered, and properties of Laguerre polynomials are exploited (somehow in the spirit of \cite{de_michele_heisenberg_1979,mller_spectral_1994,mller_marcinkiewicz_1996}). The argument presented here refines and extends the one in \cite{martini_n32}.

Theorem~\ref{thm:mhheisreit} above is just a particular case of the result presented here, and we refer the reader to the next section for a precise statement. We remark that the analogue of Theorem~\ref{thm:mhheisreit} holds on $\Heis_{\done,\dtwo}$ when the sublaplacian $L$ has the more general form
\begin{equation}\label{eq:generalsublaplacian}
L = -\sum_{\jone=1}^\done \sum_{\jtwo,\jtwo'=0}^\dtwo a^\jone_{\jtwo,\jtwo'} X_{\jtwo,\jone} X_{\jtwo',\jone}
\end{equation}
where $X_{0,\jone} = Y_\jone$ and $(a^\jone_{\jtwo,\jtwo'})_{\jtwo,\jtwo'=0,\dots,\dtwo}$ is a positive-definite symmetric matrix for all $\jone \in \{1,\dots,\done\}$. Other groups can be considered too, e.g., the complexification of a Heisenberg-Reiter group, or the quotient of the direct product of $\Heis_{1,3}$ and $N_{3,2}$ given by identifying the respective centers.

\section{The general setting}\label{section:assumption}

Let $G$ be a connected, simply connected nilpotent Lie group of step $2$. Recall that, via exponential coordinates, $G$ may be identified with its Lie algebra $\lie{g}$, that is, the tangent space of $G$ at the identity. In turn, $\lie{g}$ may be identified with the Lie algebra of left-invariant vector fields on $G$. We refer to \cite{folland_hardy_1982} for the basic definitions and further details.

Let $\lie{g}$ be decomposed as $\fst \oplus \ctr$, where $\ctr$ is the center of $\lie{g}$, and let $\langle \cdot, \cdot \rangle$ be an inner product on $\fst$. The sublaplacian $L$ associated with the inner product is defined by $L = -\sum_j X_j^2$, where $\{X_j\}_j$ is any orthonormal basis of $\fst$. Note that, vice versa, by the Poincar\'e-Birkhoff-Witt theorem, any second-order operator $L$ of the form $-\sum_j X_j^2$ for some basis $\{X_j\}_j$ of $\lie{g}$ modulo $\ctr$ determines uniquely a linear complement $\fst = \Span \{X_j\}_j$ of $\ctr$ and an inner product on $\fst$ such that $\{X_j\}_j$ is orthonormal.

Let $\ctr^*$ be the dual of $\lie{z}$ and, for all $\eta \in \ctr^*$, define $J_\eta$ as the linear endomorphism of $\fst$ such that $\eta ([z,z']) = \langle J_\eta z, z' \rangle$ for all $z,z'\in \fst$. Clearly $J_\eta$ is skewadjoint with respect to the inner product, hence $J_\eta^2$ is selfadjoint and negative semidefinite, with even rank, for all $\eta \in \ctr^*$. Set moreover $\dctr = \ctr^* \setminus \{0\}$.

\begin{assA}
There exist integers $r_1,\dots,r_\done > 0$ and an orthogonal decomposition $\fst = \fst_1 \oplus \dots \oplus \fst_\done$ such that, if $P_1,\dots,P_\done$ are the corresponding orthogonal projections, then $J_\eta P_\jone = P_\jone J_\eta$ and $J^2_\eta P_\jone$ has rank $2r_j$ and a unique nonzero eigenvalue for all $\eta \in \dctr$ and all $\jone \in \{1,\dots,\done\}$.
\end{assA}

Note that from Assumption \AssA\  it follows that $J_\eta \neq 0$ for all $\eta \in \dctr$. Therefore $[\fst,\fst] = \ctr$, that is, the decomposition $\lie{g} = \fst \oplus \ctr$ is a stratification of $\lie{g}$, and the sublaplacian $L$ is hypoelliptic.

In fact $J_\eta$ has constant rank $2(r_1 + \dots + r_k)$ for all $\eta \in \dctr$. If $J_\eta$ is invertible for all $\eta \in \dctr$, then $G$ is a M\'etivier group, and if in particular $J_\eta^2 = -|\eta|^2 \id_{\fst}$ for some inner product norm $|\cdot|$ on $\ctr^*$, then $G$ is an H-type group. The main novelty of our Assumption \AssA\  is that it allows $J_\eta$ to have a nonzero kernel when $\eta \in \dctr$, although the dimension of the kernel must be constant.

The fact that $J_\eta$ has constant rank for $\eta \in \dctr$ depends only on the algebraic structure of $G$. What depends on the inner product, that is, on the sublaplacian $L$, are the values and multiplicities of the eigenvalues of the $J_\eta$. The above Assumption \AssA\  asks for a sort of simultaneous diagonalizability of the $J_\eta$.

Under our Assumption \AssA\  on the group $G$ and the sublaplacian $L$, 
we are able to prove the following multiplier theorem.

\begin{thm}\label{thm:theorem}
Suppose that a function $F : \R \to \C$ satisfies
\[\|F \|_{MW_2^s} < \infty\]
for some $s > (\dim G)/2$. Then the operator $F(L)$ is of weak type $(1,1)$ and bounded on $L^p(G)$ for all $p \in \leftopenint 1,\infty\rightopenint$.
\end{thm}

The previously mentioned Heisenberg-Reiter groups $\Heis_{\done,\dtwo}$ satisfy Assumption \AssA, where the inner product is determined by the sublaplacian \eqref{eq:generalsublaplacian}, and the orthogonal decomposition of the first layer is given by the natural isomorphism $\R^{\dtwo \times \done} \times \R^{\done} \cong (\R^{\dtwo} \times \R)^\done$. Other examples are the free $2$-step nilpotent Lie group $N_{3,2}$ on $3$ generators, considered in \cite{martini_n32}, and its complexification $N_{3,2}^\C$. Moreover, if $G_1$ and $G_2$ satisfy Assumption \AssA, and their centers have the same dimension, then the quotient of $G_1 \times G_2$ given by any linear identification of the centers satisfy Assumption \AssA. Note that the direct product $G_1 \times G_2$ itself does not satisfy Assumption \AssA, but an adaptation of the argument presented here allows one to consider that case too. We postpone to the end of this paper a more detailed discussion of these remarks.

From now on, unless otherwise specified, we assume that $G$ and $L$ are a $2$-step stratified group and a homogeneous sublaplacian on $G$ satisfying Assumption \AssA. Since $L$ is a left-invariant operator, so is any operator of the form $F(L)$. 
Let $\Kern_{F(L)}$ denote the convolution kernel of $F(L)$. As shown, e.g., by \cite[Theorem 4.6]{martini_joint_2012}, the previous theorem is a consequence of the following estimate.

\begin{prp}\label{prp:l1estimate}
For all $s > (\dim G)/2$, there exists a weight $w_s : G \to \leftclosedint 1,\infty \rightopenint$ such that $w_s^{-1} \in L^2(G)$ and, for all compact sets $K \subseteq \R$ and for all functions $F : \R \to \C$ with $\supp F \subseteq K$,
\begin{equation}\label{eq:thel2estimate}
\|w_s \, \Kern_{F(L)}\|_2 \leq C_{K,s} \|F\|_{W_2^s};
\end{equation}
in particular
\begin{equation}\label{eq:thel1estimate}
\|\Kern_{F(L)}\|_1 \leq C_{K,s} \|F\|_{W_2^s}.
\end{equation}
\end{prp}

The rest of the paper, except for the last section, is devoted to the proof of this estimate.

\section{The joint functional calculus}

Let $\dtwo = \dim \ctr$, and let $U_1,\dots,U_\dtwo$ be any basis of the center $\ctr$. Let moreover the ``partial sublaplacian'' $L_\jone$ be defined as $L_\jone = - \sum_l X^2_{\jone,l}$, where $\{X_{\jone,l}\}_l$ is any orthonormal basis of $\fst_\jone$, for all $\jone \in \{1,\dots,\done\}$; in particular $L = L_1 + \dots + L_\done$. Then the left-invariant differential operators
\begin{equation}\label{eq:operators}
L_1,\dots,L_\done,-iU_1,\dots,-iU_\dtwo
\end{equation}
are essentially self-adjoint and commute strongly, hence they admit a joint functional calculus (see, e.g., \cite{martini_spectral_2011}). Therefore, if $\vecL$ and $\vecU$ denote the ``vectors of operators'' $(L_1,\dots,L_\done)$ and $(-iU_1,\dots,-iU_\dtwo)$, and if we identify $\ctr^*$ with $\R^\dtwo$ via the dual basis of $U_1,\dots,U_n$, then, for all bounded Borel functions $H : \R^\done \times \ctr^* \to \C$, the operator $H(\vecL,\vecU)$ is defined and bounded on $L^2(G)$. Moreover $H(\vecL,\vecU)$ is left-invariant, and we can express its convolution kernel $\Kern_{H(\vecL,\vecU)}$ in terms of Laguerre functions.

Namely, for all $n,k \in \N$, let
\[L_n^{(k)}(t) = \frac{t^{-k} e^t}{n!} \left( \frac{d}{dt} \right)^n (t^{k+n} e^{-t})\]
be the $n$-th Laguerre polynomial of type $k$, and define
\[\Ell_n^{(k)}(t) = (-1)^n e^{-t} L_n^{(k)}(2t).\]

Note that, by Assumption \AssA, for all $\eta \in \dctr$ and $\jone \in \{1,\dots,\done\}$, 
\[
J_\eta^2 P_j = -(b_\jone^\eta)^2 P_\jone^\eta
\]
for some orthogonal projection $P_\jone^\eta$ of rank $2r_\jone$ and some $b_\jone^\eta > 0$. Set moreover
\[
\bar P_\jone^\eta = P_\jone - P_\jone^\eta.
\]
Modulo reordering the $\fst_\jone$ in the decomposition of $\fst$, we may suppose that there exists $\tdone \in \{0,\dots,\done\}$ such that $\dim \fst_\jone > 2r_\jone$ if $\jone \leq \tdone$, and $\dim \fst_\jone = 2r_\jone$ if $j > \tdone$. In particular, $\bar P^\eta_\jone = 0$ and $P^\eta_\jone = P_\jone$ for all $\jone > \tdone$ and $\eta \in \dctr$.  We will also use the abbreviations $r = (r_1,\dots,r_\done)$, $\R^r = \R^{r_1} \times \dots \times \R^{r_\done}$, $\N^r = \N^{r_1} \times \dots \times \N^{r_\done}$, $|r| = r_1 + \dots + r_\done$. Moreover $\langle \cdot, \cdot \rangle$ will also denote the duality pairing $\lie{z}^* \times \lie{z} \to \R$.

\begin{prp}\label{prp:kernel}
Let $H : \R^{\done} \times \ctr^* \to \C$ be in the Schwartz class, and set 
\begin{multline}\label{eq:multiplier}
m(n,\mu,\eta) = H((2n_1+r_1) b_1^\eta + \mu_1, \dots,(2n_{\tdone} +r_{\tdone}) b_{\tdone}^\eta +\mu_{\tdone},\\
(2n_{\tdone+1} +r_{\tdone+1}) b_{\tdone+1}^\eta,\dots,(2n_\done +r_\done) b_\done^\eta, \eta)
\end{multline}
for all $n \in \N^\done$, $\mu \in \R^{\tdone}$, $\eta \in \dctr$. Then, for all $(z,u) \in G$,
\begin{multline}\label{eq:kernel}
\Kern_{H(\vecL,\vecU)}(z,u) 
 = \frac{2^{|r|}}{(2\pi)^{\dim G}} \int_{\dctr} \int_{\fst} \sum_{n \in \N^\done} m(n,(|\bar P^\eta_1 \xi|^2,\dots,|\bar P^\eta_{\tdone} \xi|^2),\eta) \\
\times \left[ \prod_{\jone=1}^\done \Ell_{n_\jone}^{(r_\jone-1)}(|P^\eta_\jone \xi|^2 /b^\eta_\jone) \right]
\, e^{i \langle \xi, z \rangle} \, e^{i \langle \eta, u \rangle} \,d\xi \,d\eta.
\end{multline}
\end{prp}
\begin{proof}
For all $\eta \in \dctr$ and $\jone \in \{1,\dots,\done\}$, let $E^\eta_{\jone,1},\bar E^\eta_{\jone,1},\dots,E^\eta_{\jone,r_\jone},\bar E^\eta_{\jone,r_\jone}$ be an orthonormal basis of the range of $P^\eta_\jone$ such that
\[J_\eta E^\eta_{\jone,l} = b^\eta_\jone \bar E^\eta_{\jone,l}, \qquad J_\eta \bar E^\eta_{\jone,l} = - b^\eta_\jone E^\eta_{\jone,l}, \qquad\text{for $l=1,\dots,r_\jone$.}\]
Hence, for all $z \in \fst$, $\eta \in \dctr$, and $\jone \in \{1,\dots,\done\}$, we can write
\[P^\eta_\jone z = \sum_{l=1}^{r_\jone} (z^\eta_{\jone,l} E^\eta_\jone + \bar z^\eta_{\jone,l} \bar E^\eta_{\jone,l})\]
for some uniquely determined $z^\eta_{\jone,l},\bar z^\eta_{\jone,l} \in \R$; set then $z^\eta_\jone = (z^\eta_{\jone,1},\dots,z^\eta_{\jone,r_\jone})$, $\bar z^\eta_\jone = (\bar z^\eta_{\jone,1},\dots,\bar z^\eta_{\jone,r_\jone})$, and moreover $z^\eta = (z^\eta_1,\dots,z^\eta_\done)$ and $\bar z^\eta = (\bar z^\eta_1,\dots,\bar z^\eta_\done)$.

For all $\eta \in \dctr$ and all $\rho \in \ker J_\eta$, an irreducible unitary representation $\pi_{\eta,\rho}$ of $G$ on $L^2(\R^r)$ is defined by
\[
\pi_{\eta,\rho}(z,u) \phi(v) = e^{i \langle \eta , u \rangle} e^{i \langle \rho, \bar P^\eta z \rangle} e^{i \sum_{\jone = 1}^\done b_\jone^\eta \langle v + z^\eta_\jone/2, \bar z^\eta_\jone \rangle}  \phi(z^\eta + v)
\]
for all $(z,u) \in G$, $v \in \R^r$, $\phi \in L^2(\R^r)$, where $\bar P^\eta = \bar P^\eta_1 + \dots + \bar P^\eta_{\tdone}$ is the orthogonal projection onto $\ker J_\eta$.
Following, e.g., \cite[\S2]{astengo_hardys_2000}, one can see that these representations are sufficient to write the Plancherel formula for the group Fourier transform of $G$, and the corresponding Fourier inversion formula:
\begin{equation}\label{eq:groupinversion}
f(z,u) = (2\pi)^{|r|-\dim G} \int_{\dctr} \int_{\ker J_\eta} \tr (\pi_{\eta,\rho}(z,u) \, \pi_{\eta,\rho}(f)) \, \prod_{\jone=1}^\done (b_\jone^\eta)^{r_\jone} \, d\rho \, d\eta
\end{equation}
for all $f : G \to \C$ in the Schwartz class and all $(z,u) \in G$, where $\pi_{\eta,\rho}(f) = \int_{G} f(g) \, \pi_{\eta,\rho}(g^{-1}) \,dg$.

Fix $\eta \in \dctr$ and $\rho \in \ker J_\eta$. The operators \eqref{eq:operators} are represented in $\pi_{\eta,\rho}$ as
\begin{equation}\label{eq:repops}
d\pi_{\eta,\rho}(L_\jone) = - \Delta_{v_\jone}^2 + (b_\jone^\eta)^2 |v_\jone|^2 + |P_\jone \rho|^2, \qquad d\pi_{\eta,\rho}(-iU_\jtwo) = \eta_\jtwo,
\end{equation}
for all $\jone \in \{1,\dots,\done\}$ and $\jtwo \in \{1,\dots,\dtwo\}$, where $v_\jone \in \R^{r_j}$ denotes the $\jone$-th component of $v \in \R^r$, and $\Delta_{v_\jone}$ denotes the corresponding partial Laplacian. Let $h_\ell$ denote the $\ell$-th Hermite function, that is,
\[
h_\ell(t) = (-1)^\ell (\ell! \, 2^\ell \sqrt{\pi})^{-1/2} e^{t^2/2} \left(\frac{d}{dt}\right)^\ell e^{-t^2},
\]
and, for all $\omega \in \N^r$, define $\tilde h_{\eta,\omega} : \R^r \to \R$ by
\[
\tilde h_{\eta,\omega} = \tilde h_{\eta,\omega,1} \otimes \dots \otimes \tilde h_{\eta,\omega,\done}, \qquad \tilde h_{\eta,\omega,\jone}(v_\jone) = (b_\jone^\eta)^{r_\jone/4} \prod_{l=1}^{r_\jone} h_{\omega_{\jone,l}}((b_\jone^\eta)^{1/2} v_{\jone,l}),
\]
for all $\jone \in \{1,\dots,\done\}$, where $\omega_{\jone,l}$ and $v_{\jone,l}$ denote the $l$-th components of $\omega_\jone \in \N^{r_\jone}$ and $v_\jone \in \R^{r_\jone}$. Then $\{\tilde h_{\eta,\omega}\}_{\omega \in \N^r}$ is a complete orthonormal system for $L^2(\R^r)$, made of joint eigenfunctions of the operators \eqref{eq:repops}. In fact,
\begin{equation}\label{eq:eigenvalues}
\begin{aligned}
d\pi_{\eta,\rho}(L_\jone) \tilde h_{\eta,\omega} &= ((2|\omega_\jone|+r_\jone)b_\jone^\eta + |P_\jone \rho|^2) \, \tilde h_{\eta,\omega}, \\
d\pi_{\eta,\rho}(-iU_\jtwo) \tilde h_{\eta,\omega} &= \eta_\jtwo \, \tilde h_{\eta,\omega},
\end{aligned}
\end{equation}
where $|\omega_\jone| = \omega_{\jone,1} + \dots + \omega_{\jone,r_\jone}$; it should be observed that $P_\jone \rho = 0$ if $\jone > \tdone$.

Since $H : \R^\done \times \ctr^* \to \C$ is in the Schwartz class, $\Kern_{H(\vecL,\vecU)} : G \to \C$ is in the Schwartz class too (see \cite[Theorem 5.2]{astengo_gelfand_2009} or \cite[\S4.2]{martini_multipliers_2010}). Moreover 
\[\pi_{\eta,\rho}(\Kern_{H(\vecL,\vecU)}) \tilde h_{\eta,\omega} = m((|\omega_1|,\dots,|\omega_\done|),(|P_1 \rho|^2,\dots,|P_{\tdone} \rho|^2),\eta) \tilde h_{\eta,\omega}\]
by \eqref{eq:eigenvalues} and \cite[Proposition 1.1]{mller_restriction_1990}; hence, if $\varphi_{\eta,\rho,\omega}(z,u) = \langle \pi_{\eta,\rho}(z,u) \tilde h_{\eta,\omega}, \tilde h_{\eta,\omega} \rangle$ is the corresponding diagonal matrix coefficient of $\pi_{\eta,\rho}$, then
\[\langle \pi_{\eta,\rho}(z,u) \, \pi_{\eta,\rho}(\Kern_{H(\vecL,\vecU)}) \tilde h_{\eta,\omega}, \tilde h_{\eta,\omega} \rangle = m((|\omega_\jone|)_{\jone \leq \done},(|P_\jone \rho|^2)_{\jone\leq \tdone},\eta) \, \varphi_{\eta,\rho,\omega}(z,u).\]
Therefore \eqref{eq:groupinversion} gives that
\begin{multline}\label{eq:quasikernel}
\Kern_{H(\vecL,\vecU)}(z,u) \\
= (2\pi)^{|r|-\dim G} \int_{\dctr} \int_{\ker J_\eta} \sum_{n \in \N^\done} m(n,(|P_\jone \rho|^2)_{\jone\leq\tdone},\eta) 
\, \psi_{\eta,\rho,n}(z,u) \prod_{\jone=1}^\done (b_\jone^\eta)^{r_\jone} \, d\rho \, d\eta,
\end{multline}
where
\[
\psi_{\eta,\rho,n}(z,u) = \sum_{\substack{\omega \in \N^r \\ |\omega_1| = n_1,\dots,|\omega_\done| = n_\done}} \varphi_{\eta,\rho,\omega}(z,u).
\]

On the other hand,
\begin{multline*}
\varphi_{\eta,\rho,\omega}(z,u) = e^{i \langle \eta , u \rangle} e^{i \langle \rho, \bar P^\eta z \rangle} \prod_{\jone=1}^\done \prod_{l=1}^{r_\jone} \Biggl[ (b_\jone^\eta)^{1/2} \\
\times \int_\R e^{i b_\jone^\eta s \bar z^\eta_{\jone,l}} \, h_{\omega_{\jone,l}}((b_\jone^\eta)^{1/2} (s+z^\eta_{\jone,l}/2)) \, h_{\omega_{\jone,l}}((b_\jone^\eta)^{1/2} (s-z^\eta_{\jone,l}/2)) \,ds \Biggr].
\end{multline*}
The last integral is essentially the Fourier-Wigner transform of a pair of Hermite functions, whose bidimensional Fourier transform is a Fourier-Wigner transform too \cite[formula (1.90)]{folland_harmonic_1989}. The parity properties of the Hermite functions then yield
\begin{multline*}
\varphi_{\eta,\rho,\omega}(z,u) = e^{i \langle \eta , u \rangle} e^{i \langle \rho, \bar P^\eta z \rangle} \prod_{\jone=1}^\done  \prod_{l=1}^\jone \Biggl[ \frac{(-1)^{\omega_{\jone,l}}}{\pi \, b_\jone^\eta} \int_{\R \times \R} e^{i \theta_1 z^\eta_{\jone,l}} e^{i \theta_2 \bar z^\eta_{\jone,l}} \\
\times \int_\R e^{it (2\theta_1/ (b_\jone^\eta)^{1/2}) } \, h_{\omega_{\jone,l}}(t+\theta_2/(b_\jone^\eta)^{1/2}) \, h_{\omega_{\jone,l}}(t-\theta_2/(b_\jone^\eta)^{1/2}) \,dt \,d\theta_1 \,d\theta_2 \Biggr].
\end{multline*}
Since the Fourier-Wigner transform of a pair of Hermite functions can be expressed in terms of Laguerre polynomials (see \cite[Theorem 1.104]{folland_harmonic_1989} or \cite[Theorem 1.3.4]{thangavelu_lectures_1993}), we obtain that
\begin{multline*}
\varphi_{\eta,\rho,\omega}(z,u) = \frac{e^{i \langle \eta , u \rangle} e^{i \langle \rho, \bar P^\eta z \rangle}}{\pi^{|r|}} \int_{\R^r \times \R^r}  e^{i \langle \zeta_1, z^\eta \rangle} e^{i \langle \zeta_2, \bar z^\eta \rangle} \\
\times \prod_{\jone=1}^\done \Biggl[ (b_\jone^\eta)^{-r_\jone} \prod_{l=1}^{r_\jone} \Ell_{\omega_{\jone,l}}^{(0)}((\zeta_{1,\jone,l}^2 + \zeta_{2,\jone,l}^2)/b^\eta_\jone) \Biggr] \,d\zeta_1 \,d\zeta_2
\end{multline*}
Consequently, for all $n \in \N^\done$,
\begin{multline}\label{eq:diagonalcoefficients}
\psi_{\eta,\rho,n}(z,u) = \frac{e^{i \langle \eta , u \rangle} e^{i \langle \rho, \bar P^\eta z \rangle}}{\pi^{|r|}} 
\int_{\R^r \times \R^r}  e^{i \langle \zeta_1, z^\eta \rangle} e^{i \langle \zeta_2, \bar z^\eta \rangle} \\
\times \prod_{\jone=1}^\done \Biggl[ (b_\jone^\eta)^{-r_\jone} \, 2^{r_\jone-1} \Ell_{n_\jone}^{(r_\jone-1)}((|\zeta_{1,\jone}|^2 + |\zeta_{2,\jone}|^2)/b^\eta_\jone) \Biggr] \,d\zeta_1 \,d\zeta_2
\end{multline}
\cite[\S 10.12, formula (41)]{erdelyi_higher2_1981}.
The conclusion then follows by plugging \eqref{eq:diagonalcoefficients} into \eqref{eq:quasikernel} and performing a change of variable by rotation in the inner integrals.
\end{proof}

\section{A weighted Plancherel estimate}

Proposition~\ref{prp:kernel} expresses the convolution kernel $\Kern_{H(\vecL,\vecU)}$ as the inverse Fourier transform of a function of the multiplier $H$. Due to the properties of the Fourier transform, it is not unreasonable to think that multiplying the kernel by a polynomial weight might correspond to taking derivatives of the multiplier. As a matter of fact, the presence of the Laguerre expansion leads us to consider both ``discrete'' and ``continuous'' derivatives of the reparametrization $m : \N^\done \times \R^{\tdone} \times \dctr \to \C$ of the multiplier $H$ given by \eqref{eq:multiplier}.

For convenience, set $\Ell_n^{(k)} = 0$ for all  $n < 0$. From the properties of Laguerre polynomials (see, e.g., \cite[\S10.12]{erdelyi_higher2_1981}) one can easily derive the following identities.

\begin{lem}
For all $k,n,m \in \N$ and $t \in \R$,
\begin{gather}
\label{eq:laguerrepm} \Ell_n^{(k)}(t) = \Ell_{n-1}^{(k+1)}(t) + \Ell_n^{(k+1)}(t), \\
\label{eq:laguerred} \frac{d}{dt} \Ell_n^{(k)}(t) = \Ell_{n-1}^{(k+1)}(t) - \Ell_{n}^{(k+1)}(t),\\
\label{eq:laguerreo} \int_{0}^{\infty} \Ell_n^{(k)}(t) \, \Ell_m^{(k)}(t) \, t^k \,dt = \begin{cases}
\frac{(n+k)!}{2^{k+1} n!} &\text{if $n=m$,}\\
0 &\text{otherwise.}
\end{cases}
\end{gather}
\end{lem}

Let $e_1,\dots,e_\done$ denote the standard basis of $\R^\done$. We introduce some operators on functions $f : \N^\done \times \R^{\tdone} \times \dctr \to \C$:
\begin{align*}
\tau_\jone f(n,\mu,\eta) &= f(n+e_\jone,\mu,\eta), \\
\delta_\jone f(n,\mu,\eta) &= f(n+e_\jone,\mu,\eta) - f(n,\mu,\eta), \\
\partial_{\mu_l} f(n,\mu,\eta) &= \frac{\partial}{\partial \mu_l} f(n,\mu,\eta), \\
\partial_{\eta_\jtwo} f(n,\mu,\eta) &= \frac{\partial}{\partial \eta_\jtwo} f(n,\mu,\eta)
\end{align*}
for all $\jone \in \{1,\dots,\done\}$, $l \in \{1,\dots,\tdone\}$, $\jtwo \in \{1,\dots,\dtwo\}$.

For all $h \in \N$ and all multiindices $\alpha \in \N^h$, we denote by $|\alpha|$ the length $\alpha_1+\dots+\alpha_h$ of $\alpha$. Inequalities between multiindices, such as $\alpha \leq \alpha'$, shall be interpreted componentwise. Set moreover $(\alpha)_+ = ((\alpha_1)_+,\dots,(\alpha_h)_+)$, where $(\ell)_+ = \max\{\ell,0\}$.

A function $\Psi : \dctr \times \fst \to \C$ will be called \emph{multihomogeneous} if there exist $h_0,h_1,\dots,h_\done \in \R$ such that
\[\Psi\Biggl(\lambda_0 \eta,\sum_{\jone=1}^\done \lambda_\jone P_\jone \xi\Biggr) = \lambda_0^{h_0} \lambda_1^{h_1} \dots \lambda_\done^{h_\done} \Psi(\eta,\xi)\]
for all $\eta \in \dctr$, $\xi \in \fst$, $\lambda_0,\lambda_1,\dots,\lambda_\done \in \leftopenint 0,\infty \rightopenint$; the degrees of homogeneity $h_0,h_1,\dots,h_\done$ of $\Psi$ will also be denoted as $\deg_\ctr \Psi, \deg_{\fst_1} \Psi, \dots, \deg_{\fst_\done} \Psi$. Note that, if $\Psi$ is multihomogeneous and continuous, then $\deg_{\fst_\jone} \Psi \geq 0$ for all $\jone \in \{1,\dots,\done\}$.

\begin{prp}\label{prp:weightedkernel}
Let $H : \R^\done \times \ctr^* \to \C$ be smooth and compactly supported in $\R^\done \times \dctr$, and let $m(n,\mu,\eta)$ be defined by \eqref{eq:multiplier}.
For all $\alpha \in \N^\dtwo$,
\begin{multline*}
u^\alpha \, \Kern_{H(\vecL,\vecU)}(z,u) = \sum_{\iota \in I_\alpha} \int_{\dctr} \int_{\fst} \sum_{n \in \N^\done} \partial_\eta^{\gamma^\iota} \partial_\mu^{\theta^\iota} \delta^{\beta^\iota} m(n,(|\bar P^\eta_\jone \xi|^2)_{\jone\leq\tdone},\eta) \\
\times \Psi_\iota(\eta,\xi) \, \Biggl[ \prod_{\jone=1}^\done \Ell^{(r_\jone-1 + \beta^\iota_\jone)}_{n_\jone}(|P^\eta_\jone \xi|^2/b^\eta_\jone) \Biggr] \, e^{i\langle \xi,z \rangle} \, e^{i\langle \eta,u\rangle} \,d\xi \,d\eta,
\end{multline*}
for almost all $(z,u) \in G$, where $I_\alpha$ is a finite set and, for all $\iota \in I_\alpha$,
\begin{itemize}
\item $\gamma^\iota \in \N^\dtwo$, $\theta^\iota \in \N^{\tdone}$, $\beta^\iota \in \N^\done$, $\gamma^\iota \leq \alpha$,
\item $\Psi_\iota = \Psi_{\iota,0} \Psi_{\iota,1} \dots \Psi_{\iota,\done}$, where $\Psi_{\iota,\jone} : \dctr \times \fst \to \C$ is smooth and multihomogeneous for all $\jone \in \{0,\dots,\done\}$,
\item $\deg_\ctr \Psi_\iota = |\gamma^\iota| - |\alpha| - |\beta^\iota|$ and $\deg_{\fst_\jone} \Psi_\iota = 2\beta^\iota_\jone + 2\theta^\iota_\jone$ for all $\jone \in \{1,\dots,\done\}$,
\item for all $\jone \in \{1,\dots,\done\}$, $\Psi_{\iota,\jone}(\eta,\xi)$ is a product of factors of the form $|P_\jone^\eta \xi|^2$ or $\partial_{\eta_\jtwo} |P^\eta_\jone \xi|^2$ for $\jtwo \in \{1,\dots,\dtwo\}$,
\item $|\gamma^\iota| + |\theta^\iota| + |\beta^\iota| + \sum_{\jone=1}^{\done} (\beta^\iota_\jone - (\deg_{\fst_\jone} \Psi_{\iota,\jone}) /2)_+ \leq |\alpha|$.
\end{itemize}
\end{prp}
\begin{proof}
By Proposition~\ref{prp:kernel} and the properties of the Fourier transform, we are reduced to proving that, for all $\alpha \in \N^{\dtwo}$, $\eta \in \dctr$, $\xi \in \fst$,
\begin{multline*}
\left(\frac{\partial}{\partial \eta}\right)^\alpha \sum_{n \in \N^\done} m(n,(|\bar P^\eta_\jone \xi|^2)_{\jone\leq\tdone},\eta) \, \prod_{\jone=1}^\done \Ell_{n_\jone}^{(r_\jone-1)}(|P^\eta_\jone \xi|^2 /b^\eta_\jone) \\
= \sum_{\iota \in I_\alpha} \sum_{n \in \N^\done} \partial_\eta^{\gamma^\iota} \partial_\mu^{\theta^\iota} \delta^{\beta^\iota} m(n,(|\bar P^\eta_\jone \xi|^2)_{\jone\leq\tdone},\eta)
\, \Psi_\iota(\eta,\xi) 
\, \prod_{\jone=1}^\done \Ell^{(r_\jone-1 + \beta^\iota_\jone)}_{n_\jone}(|P^\eta_\jone \xi|^2/b^\eta_\jone),
\end{multline*}
where $I_\alpha$, $\gamma^\iota$, $\theta^\iota$, $\beta^\iota$, $\Psi_\iota$ are as in the above statement.

This is easily proved by induction on $|\alpha|$. For $|\alpha| = 0$ it is trivially verified. For the inductive step, one applies Leibniz' rule, and exploits the following observations:
\begin{itemize}
\item when a derivative $\partial_{\eta_\jtwo}$ hits a Laguerre function, by the identity \eqref{eq:laguerred} and summation by parts, the type of the Laguerre function is increased by $1$, as well as the corresponding component of $\beta^\iota$;
\item for all $\jone \in \{1,\dots,\done\}$, $b^\eta_\jone = \sqrt{\tr(-J_\eta^2 P_\jone)/(2r_\jone)}$ is a smooth function of $\eta \in \dctr$, homogeneous of degree $1$;
\item for all $\jone \in \{1,\dots,\done\}$, $P^\eta_\jone = -J_\eta^2 P_j/(b^\eta_\jone)^2$ is a smooth function of $\eta \in \dctr$, homogeneous of degree $0$, and in fact it is constant if $\jone > \tdone$;
\item for all $\jone \in \{1,\dots,\tdone\}$, $|P^\eta_\jone \xi|^2 = \langle P^\eta_\jone P_\jone \xi, P_\jone \xi \rangle$ is a smooth bihomogeneous function of $(\eta,P_\jone\xi) \in \dctr \times \fst_\jone$ of bidegree $(0,2)$, and moreover
\begin{gather*}
|\bar P^\eta_\jone \xi|^2 = |P_\jone \xi|^2 - |P^\eta_\jone \xi|^2, \qquad \partial_{\eta_\jtwo} |\bar P^\eta_\jone \xi|^2 = - \partial_{\eta_\jtwo} |P^\eta_\jone \xi|^2, \\
\partial_{\eta_\jtwo} (|P^\eta_\jone \xi|^2/b^\eta_\jone) = |P^\eta_\jone \xi|^2 \partial_{\eta_\jtwo} (1/b^\eta_\jone) + (\partial_{\eta_\jtwo} |P^\eta_\jone \xi|^2)/b^\eta_\jone
\end{gather*}
for all $\jtwo \in \{1,\dots,\dtwo\}$.
\end{itemize}
The conclusion follows.
\end{proof}

Note that, for all $\jone \in \{1,\dots,\done\}$, $\mu \in \R^{\tdone}$, $\eta \in \dctr$, the quantities $\tau_\jone f(\cdot,\mu,\eta)$, $\delta_\jone f(\cdot,\mu,\eta)$ depend only on $f(\cdot,\mu,\eta)$; in other words, $\tau_\jone$ and $\delta_\jone$ can be considered as operators on functions $\N^\done \to \C$.

The following lemma exploits the orthogonality properties \eqref{eq:laguerreo} of the Laguerre functions, together with \eqref{eq:laguerrepm}, and shows that a mismatch between the type of the Laguerre function and the exponent of the weight attached to the measure may be turned in some cases into discrete differentiation.

\begin{lem}\label{lem:laguerreorthogonality}
For all $h,k \in \N^\done$ and all compactly supported $f : \N^\done \to \C$,
\begin{multline*}
\int_{\leftopenint 0,\infty\rightopenint^\done} \Bigl| \sum_{n \in \N^\done} f(n) \, \prod_{\jone=1}^\done \Ell_{n_\jone}^{(k_\jone)}(t_\jone) \Bigr|^2 \,t^h \,dt \\
\leq C_{h,k} \sum_{n \in \N^\done} |\delta^{(k-h)_+} f(n)|^2 \, \prod_{\jone=1}^\done (1+ n_\jone)^{h_\jone+2(k_\jone-h_\jone)_+}.
\end{multline*}
\end{lem}
\begin{proof}
Via an inductive argument, we may reduce to the case $\done = 1$.

Note that, if $f$ is compactly supported, then $\tau^l f$ is null for all sufficiently large $l \in \N$. Hence the operator $1+\tau$, when restricted to the set of compactly supported functions, is invertible, with inverse given by
\[
(1+\tau)^{-1} f = \sum_{l \in \N} (-1)^l \tau^l f.
\]
Then by \eqref{eq:laguerrepm} we deduce that, for all $k \in \N$,
\begin{align*}
\sum_{n \in \N} f(n) \, \Ell_n^{(k)}(t) &= \sum_{n \in \N} (1+\tau) f(n) \, \Ell_n^{(k+1)}(t), \\
\sum_{n \in \N} f(n) \, \Ell_n^{(k+1)}(t) &= \sum_{n \in \N} (1+\tau)^{-1} f(n) \, \Ell_n^{(k)}(t),
\end{align*}
and consequently, for all $h,k \in \N$,
\[
\sum_{n \in \N} f(n) \, \Ell_n^{(k)}(t) = \sum_{n \in \N} (1+\tau)^{h-k} f(n) \, \Ell_n^{(h)}(t)
\]
Thus the orthogonality properties \eqref{eq:laguerreo} of the Laguerre functions give us that
\[
\int_0^\infty \Bigl| \sum_{n \in \N} f(n) \, \Ell_n^{(k)}(t) \Bigr|^2 \,t^h \,dt \leq C_{h,k} \sum_{n \in \N} |(1+\tau)^{h-k} f(n)|^2 \, \langle n \rangle^{h},
\]
where $\langle n \rangle = 1+n$.

In the case $h \geq k$, $(1+\tau)^{h-k}$ is given by the finite sum
\[(1+\tau)^{h-k} = \sum_{\ell = 0}^{h-k} \binom{h-k}{\ell} \tau^\ell,\]
and the conclusion follows immediately by the triangular inequality.

In the case $h < k$, instead, since $\delta = \tau - 1$, from the identity $1-\tau^2 = (1-\tau)(1+\tau)$ we deduce that
\[(1+\tau)^{h-k} = (-\delta)^{k-h} (1-\tau^2)^{h-k} = (-1)^{k-h} \sum_{\ell \geq 0} \binom{\ell + k - h - 1}{\ell} \delta^{k-h} \tau^{2\ell},\]
hence
\[\begin{split}
\sum_{n \in \N} |(1+\tau)^{h-k} f(n)|^2 \, \langle n \rangle^{h} &= \sum_{n \in \N} \Biggl|\sum_{\ell \geq 0} \binom{\ell + k - h - 1}{\ell} \delta^{k-h} f(n+2\ell)\Biggr|^2 \, \langle n \rangle^{h} \\
&\leq C_{h,k} \sum_{n \in \N} \Biggl|\sum_{\ell \geq n} \langle\ell\rangle^{k - h - 1} \delta^{k-h} f(\ell)\Biggr|^2 \, \langle n \rangle^{h} \\
&\leq C_{h,k} \sum_{n \in \N} \langle n \rangle^{-1/2} \sum_{\ell \geq n} |\langle\ell\rangle^{k - h - 1/4} \delta^{k-h} f(\ell)|^2 \, \langle n \rangle^{h} \\
&\leq C_{h,k} \sum_{\ell \in \N} \langle\ell\rangle^{2k - 2h - 1/2} |\delta^{k-h} f(\ell)|^2 \sum_{n = 0}^\ell \langle n \rangle^{h-1/2} \\
&\leq C_{h,k} \sum_{\ell \in \N} \langle\ell\rangle^{2k - h} |\delta^{k-h} f(\ell)|^2,
\end{split}\]
by the Cauchy-Schwarz inequality, and we are done.
\end{proof}

Let $|\cdot|$ denote any Euclidean norm on $\lie{z}^*$. The previous lemma, together with Plancherel's formula for the Fourier transform, yields the following $L^2$ estimate.

\begin{prp}\label{prp:generalweightedl2}
Under the hypotheses of Proposition~\ref{prp:weightedkernel}, for all $\alpha \in \N^\dtwo$,
\begin{multline}\label{eq:weightedl2}
\int_{G} |u^\alpha \, \Kern_{H(\vecL,\vecU)}(z,u)|^2 \,dz\,du 
\leq C_\alpha \sum_{\iota \in \tilde I_\alpha} \int_{\dctr} \int_{\leftclosedint 0,\infty\rightopenint^{\tdone}}  \sum_{n \in \N^\done} |\partial_\eta^{\gamma^\iota} \partial_\mu^{\theta^\iota} \delta^{\beta^\iota} m(n,\mu,\eta)|^2 \\
\times |\eta|^{2|\gamma^\iota|-2|\alpha|-2|\beta^\iota|+|a^\iota|+\done} \, (1+n_1)^{a^\iota_1} \dots (1+n_\done)^{a^\iota_\done} \, d\sigma_\iota(\mu) \,d\eta,
\end{multline}
where $\tilde I_\alpha$ is a finite set and, for all $\iota \in \tilde I_\alpha$,
\begin{itemize}
\item $\gamma^\iota \in \N^\dtwo$, $\theta^\iota \in \N^{\tdone}$, $a^\iota, \beta^\iota \in \N^\done$,
\item $\gamma^\iota \leq \alpha$, $|\gamma^\iota| + |\theta^\iota| + |\beta^\iota| \leq |\alpha|$,
\item $\sigma_\iota$ is a regular Borel measure on $\leftclosedint 0,\infty\rightopenint^{\tdone}$.
\end{itemize}
\end{prp}
\begin{proof}
Note that, for all $\jone \in \{1,\dots,\done\}$,
\[\partial_{\eta_\jtwo} (|P^\eta_\jone \xi|^2) = 2 \langle (\partial_{\eta_\jtwo} P^\eta_\jone) P_\jone \xi, P^\eta_\jone \xi \rangle \leq C|\eta|^{-1} |P^\eta_\jone \xi| |P_\jone \xi|;\]
consequently, if $\Psi_\iota,\Psi_{\iota,\jone},\gamma^\iota,\theta^\iota,\beta^\iota$ are as in the statement of Proposition~\ref{prp:weightedkernel}, then
\[|\Psi_{\iota,\jone}(\eta,\xi)|^2 \leq C_\iota |\eta|^{2\deg_\ctr \Psi_{\iota,\jone}} |P^\eta_\jone \xi|^{\deg_{\fst_\jone} \Psi_{\iota,\jone}} |P_\jone \xi|^{\deg_{\fst_\jone} \Psi_{\iota,\jone}}\]
for all $\jone \in \{1,\dots,\done\}$, hence
\[\begin{split}
|\Psi_\iota(\eta,\xi)|^2 &\leq C_\iota |\eta|^{2\deg_\ctr \Psi_\iota} \prod_{\jone=1}^{\done} |P^\eta_\jone \xi|^{\deg_{\fst_\jone} \Psi_{\iota,\jone}} |P_\jone \xi|^{2\deg_{\fst_\jone} \Psi_{\iota,0} + \deg_{\fst_\jone} \Psi_{\iota,\jone}} \\
&\leq  C_\iota |\eta|^{2|\gamma^\iota|-2|\alpha|-2|\beta^\iota|} \prod_{\jone=1}^{\done} \sum_{h_\jone = (\deg_{\fst_\jone} \Psi_{\iota,\jone})/2}^{2\theta^\iota_\jone + 2\beta^\iota_\jone} |P^\eta_\jone \xi|^{2h_\jone} |\bar P^\eta_\jone \xi|^{4\theta^\iota_\jone + 4\beta^\iota_\jone-2h_\jone},
\end{split}\]
and moreover, for all $h \in \N^\done$, if $h_j \geq (\deg_{\fst_\jone} \Psi_{\iota,\jone})/2$ for all $j \in \{1,\dots,\done\}$, then
\[|\gamma^\iota| + |\theta^\iota| + |\beta^\iota| + \sum_{\jone=1}^{\tdone} (\beta^\iota_\jone - h_\jone)_+ \leq |\alpha|.\]

By Proposition~\ref{prp:weightedkernel}, Plancherel's formula and the triangular inequality, we then obtain that the left-hand side of \eqref{eq:weightedl2} is majorized by a finite sum of terms of the form
\begin{multline}\label{eq:generalterm}
\int_{\dctr} \int_{\fst} \Biggl| \sum_{n \in \N^\done} \partial_\eta^{\gamma} \partial_\mu^{\theta} \delta^{\beta} m(n,(|\bar P^\eta_\jone \xi|^2)_{\jone\leq\tdone},\eta) 
\, \prod_{\jone=1}^\done \Ell^{(r_\jone-1+\beta_\jone)}_{n_\jone}(|P^\eta_\jone \xi|^2/b^\eta_\jone) \Biggr|^2 \\
\times |\eta|^{2|\gamma|-2|\alpha|-2|\beta|} \, \prod_{\jone=1}^\done |P^\eta_\jone \xi|^{2h_\jone} \prod_{\jone=1}^{\tdone} |\bar P^\eta_\jone \xi|^{2k_\jone} \,d\xi \,d\eta,
\end{multline}
where $\gamma \in \N^\dtwo$, $\theta,k \in \N^{\tdone}$, $\beta,h \in \N^\done$ and $|\gamma|+|\theta|+|\beta + (\beta-h)_+| \leq |\alpha|$. Simple changes of variables (rotation, polar coordinates and rescaling) allow one to rewrite \eqref{eq:generalterm} as a constant times
\begin{multline*}
\int_{\dctr} \int_{\leftopenint 0,\infty\rightopenint^{\tdone}} \int_{\leftopenint 0,\infty\rightopenint^{\done}} \Biggl| \sum_{n \in \N^\done} \partial_\eta^{\gamma} \partial_\mu^{\theta} \delta^{\beta} m(n,\mu,\eta) 
\, \prod_{\jone=1}^\done \Ell^{(r_\jone-1+\beta_\jone)}_{n_\jone}(t_\jone) \Biggr|^2 \, \prod_{\jone=1}^\done t_\jone^{r_\jone-1+h_\jone} \,dt \\
\times |\eta|^{2|\gamma|-2|\alpha|-2|\beta|} \prod_{\jone=1}^\done (b^\eta_\jone)^{h_\jone+r_\jone} \prod_{\jone=1}^{\tdone} \mu_\jone^{k_\jone+(\dim \fst_\jone-2r_\jone)/2} \,\frac{d\mu}{\mu_1 \cdots \mu_{\tdone}} \,d\eta.
\end{multline*}
By exploiting the fact that the $b^\eta_\jone$ are smooth functions of $\eta \in \dctr$, homogeneous of degree $1$ (see the proof of Proposition~\ref{prp:weightedkernel}), and applying Lemma~\ref{lem:laguerreorthogonality} to the inner integral, the last quantity is majorized by
\begin{multline*}
C \int_{\dctr} \int_{\leftopenint 0,\infty\rightopenint^{\tdone}} \sum_{n \in \N^\done} | \partial_\eta^{\gamma} \partial_\mu^{\theta} \delta^{\beta+(\beta- h)_+} m(n,\mu,\eta) 
|^2 \, \prod_{\jone=1}^\done (1+ n_\jone)^{r_\jone -1 + h_\jone+2(\beta_\jone-h_\jone)_+} \\
\times |\eta|^{2|\gamma|-2|\alpha|-2|\beta|+ |h|+|r|} \prod_{\jone=1}^{\tdone} \mu_\jone^{k_\jone+(\dim \fst_\jone-2r_\jone)/2} \,\frac{d\mu}{\mu_1 \dots \mu_{\tdone}} \,d\eta,
\end{multline*}
and since the exponents $k_\jone+(\dim \fst_\jone-2r_\jone)/2$ are strictly positive, while
\[
-2|\beta| + |h| + |r| = -2|\beta+(\beta-h)_+| + \sum_{\jone=1}^\done (r_\jone-1+h_\jone+2(\beta_\jone-h_\jone)_+) + \done
\]
and $|\gamma|+|\theta|+|\beta + (\beta-h)_+| \leq |\alpha|$, the conclusion follows by suitably renaming the multiindices.
\end{proof}

\section{From discrete to continuous}

Via the fundamental theorem of integral calculus, finite differences can be estimated by continuous derivatives. The next lemma is a multivariate analogue of \cite[Lemma~6]{martini_n32}, and we omit the proof (see also \cite[Lemma~7]{martini_grushin2}).

\begin{lem}\label{lem:discretecontinuous}
Let $f : \N^\done \to \C$ have a smooth extension $\tilde f : \leftclosedint 0,\infty \rightopenint^\done \to \C$, and let $\beta \in \N^\done$. Then
\[\delta^\beta f(n) = \int_{J_\beta} \partial^\beta \tilde f(n+s) \,d\nu_\beta(s)\]
for all $n \in \N$, where $J_\beta = \prod_{\jone=1}^\done \leftclosedint 0,\beta_\jone\rightclosedint$ and $\nu_\beta$ is a Borel probability measure on $J_\beta$. In particular
\[|\delta^\beta f(n)|^2 \leq \int_{J_\beta} |\partial^\beta \tilde f(n+s)|^2 \,d\nu_\beta(s)\]
for all $n \in \N^\done$.
\end{lem}

We give now a simplified version of the right-hand side of \eqref{eq:weightedl2}, in the case we restrict to the functional calculus of $L$ alone. In order to avoid issues of divergent series, it is however convenient at first to truncate the multiplier along the spectrum of $\vecU$.

\begin{lem}\label{lem:weightedplancherel}
Let $\chi \in C^\infty_c(\R)$ be supported in $\leftclosedint 1/2,2 \rightclosedint$, $K \subseteq \R$ be compact and $M \in \leftopenint 0, \infty \rightopenint$. If $F : \R \to \C$ is smooth and supported in $K$, and $F_M : \R \times \ctr^* \to \C$ is given by
\[F_M(\lambda,\eta) = F(\lambda) \, \chi(|\eta|/M),\]
then, for all $r \in \leftclosedint 0, \infty \rightopenint$,
\[\int_{G} | |u|^r \, \Kern_{F_M(L,\vecU)}(z,u) |^2 \,dz \,du \leq C_{K,\chi,r} \, M^{\dtwo-2r} \|F\|_{W_2^r}^2.\]
\end{lem}
\begin{proof}
We may restrict to the case $r \in \N$, the other cases being recovered a posteriori by interpolation. Hence we need to prove that
\begin{equation}\label{eq:weightedmultiindex}
\int_{G} | u^\alpha \, \Kern_{F_M(L,\vecU)}(z,u) |^2 \,dz \,du \leq C_{K,\chi,\alpha} \, M^{\dtwo-2|\alpha|} \|F\|_{W_2^{|\alpha|}}^2
\end{equation}
for all $\alpha \in \N^{\done}$. On the other hand, if $m$ is defined by
\begin{equation}\label{eq:simplermultiplier}
m(n,\mu,\eta) = F\Biggl( \sum_{\jone=1}^\done b^\eta_\jone \langle n_\jone \rangle_\jone + |\mu|_\Sigma \Biggr) \, \chi(|\eta|/M),
\end{equation}
where $\langle \ell \rangle_{\jone} = 2\ell + r_\jone$ and $|\mu|_\Sigma = \sum_{\jone=1}^\tdone \mu_\jone$, then the left-hand side of \eqref{eq:weightedmultiindex} is majorized by the right-hand side of \eqref{eq:weightedl2}, and we are reduced to prove
\begin{multline}\label{eq:weightedsummand}
\sum_{n \in \N^\done} \int_{\dctr} \int_{\leftclosedint 0,\infty\rightopenint^{\tdone}}  |\partial_\eta^{\gamma^\iota} \partial_\mu^{\theta^\iota} \delta^{\beta^\iota} m(n,\mu,\eta)|^2 
\, |\eta|^{2|\gamma^\iota|-2|\alpha|-2|\beta^\iota|+|a^\iota|+\done} \\
\times (1+n_1)^{a^\iota_1} \dots (1+n_\done)^{a^\iota_\done} 
\, d\sigma_\iota(\mu) \,d\eta \leq C_{K,\chi,\alpha} \, M^{\dtwo-2|\alpha|} \|F\|_{W_2^{|\alpha|}}^2
\end{multline}
for all $\iota \in \tilde I_\alpha$, where $\tilde I_\alpha$, $\gamma^\iota$, $\theta^\iota$, $\beta^\iota$, $a^\iota$, $\sigma_\iota$ are as in Proposition~\ref{prp:generalweightedl2}.

Note that the right-hand side of \eqref{eq:simplermultiplier} makes sense for all $n \in \R^\done$, and defines a smooth extension of $m$, which we still denote by $m$ by a slight abuse of notation. Hence, by Lemma~\ref{lem:discretecontinuous},
\begin{equation}\label{eq:discretecontinuous}
|\partial_\eta^{\gamma_\iota} \partial_\mu^{\theta^\iota} \delta^{\beta^\iota} m(n,\mu,\eta)|^2 \leq \int_{J_\iota} |\partial_{\eta}^{\gamma^\iota} \partial_\mu^{\theta^\iota} \partial_n^{\beta^\iota} m(n+s,\mu,\eta)|^2 \,d\nu_\iota(s),
\end{equation}
where $J_\iota = \prod_{\jone=1}^\done \leftclosedint 0,\beta^\iota_\jone \rightclosedint$ and $\nu_\iota$ is a suitable probability measure on $J_\iota$. Moreover the measure $\sigma_\iota$ in \eqref{eq:weightedsummand} is finite on compacta, and the right-hand side of \eqref{eq:discretecontinuous} vanishes when $|\mu|_\Sigma > \max K$, because $\supp F \subseteq K$. Consequently \eqref{eq:weightedsummand} will be proved if we show that
\begin{multline}\label{eq:weightedsummand2}
\sum_{n \in \N^\done} \int_{\dctr} |\partial_\eta^{\gamma^\iota} \partial_\mu^{\theta^\iota} \partial_n^{\beta^\iota} m(n+s,\mu,\eta)|^2 
\, |\eta|^{2|\gamma^\iota|-2|\alpha|-2|\beta^\iota|+|a^\iota|+\done} \\
\times (1+ n_1)^{a^\iota_1} \dots (1+n_\done)^{a^\iota_\done} 
 \,d\eta \leq C_{K,\chi,\alpha} \, M^{\dtwo-2|\alpha|} \|F\|_{W_2^{|\alpha|}}^2
\end{multline}
for all $s \in J_\iota$ and $\mu \in \leftclosedint 0,\max K \rightclosedint^{\tdone}$, uniformly in $s$ and $\mu$.

As observed in the proof of Proposition~\ref{prp:weightedkernel}, the $b_\jone^\eta$ are positive, smooth functions of $\eta \in \dctr$, homogeneous of degree $1$; therefore, for all $n \in \N^\done$, $\jone \in \{1,\dots,\done\}$, $\eta \in \dctr$, $s \in \leftclosedint 0,\infty \rightopenint^\done$, $\mu \in \leftclosedint 0,\infty \rightopenint^{\tdone}$,
\begin{equation}\label{eq:upperbound_eta_n}
|\eta| (1+ n_\jone) \sim b_\jone^\eta \langle n_\jone \rangle_\jone \leq \sum_{l=1}^\done  b^\eta_l \langle n_l + s_l \rangle_l + |\mu|_\Sigma,
\end{equation}
and the last quantity is bounded by the constant $\max K$ whenever $(n+s,\mu,\eta) \in \supp m$, because $\supp F \subseteq K$. Hence the factors $|\eta| (1+ n_\jone)$ in the left-hand side of \eqref{eq:weightedsummand2} can be discarded, that is, we are reduced to proving \eqref{eq:weightedsummand2} in the case $a^\iota = 0$.

From \eqref{eq:simplermultiplier} it is easily proved inductively that
\begin{multline*}
\partial_\eta^{\gamma^\iota} \partial_\mu^{\theta^\iota} \partial_n^{\beta^\iota} m(n,\mu,\eta) 
= \sum_{\substack{\upsilon \in \N^\done \\ |\upsilon| \leq |\gamma^\iota|}} \sum_{q=0}^{|\gamma^\iota| - |\upsilon|} F^{(|\theta^\iota|+|\beta^\iota|+|\upsilon|)}\Biggl(\sum_{\jone=1}^\done b_\jone^\eta \langle n_\jone \rangle_\jone + |\mu|_\Sigma\Biggr) \\
\times \Psi_{\iota,\upsilon,q}(\eta) \, M^{-q} \, \chi^{(q)}(|\eta|/M) \prod_{\jone=1}^\done \langle n \rangle_\jone^{\upsilon_\jone} 
\end{multline*}
where $\Psi_{\iota,\upsilon,q} : \dctr \to \R$ is smooth and homogeneous of degree $|\beta^\iota| + |\upsilon| + q - |\gamma^\iota|$. By exploiting again \eqref{eq:upperbound_eta_n} and the fact that $\supp F \subseteq K$, we than obtain that
\begin{multline*}
|\partial_\eta^{\gamma^\iota} \partial_\mu^{\theta^\iota} \partial_n^{\beta^\iota} m(n,\mu,\eta)|^2 
\leq C_{K,\chi,\alpha} M^{2|\beta^\iota| - 2|\gamma^\iota|} \tilde\chi(|\eta|/M) \\
\times \sum_{v=0}^{|\gamma^\iota|}  \Biggl|F^{(|\beta^\iota|+|\theta^\iota|+v)}\Biggl(\sum_{\jone=1}^\done b_\jone^\eta \langle n_\jone \rangle_\jone + |\mu|_\Sigma\Biggr)\Biggr|^2,
\end{multline*}
where $\tilde\chi$ is the characteristic function of $\leftclosedint 1/2,2\rightclosedint$. Hence the left-hand side of \eqref{eq:weightedsummand2}, when $a^\iota = 0$, is majorized by
\begin{multline*}
C_{K,\chi,\alpha} M^{\done -2|\alpha|} \\
\times \sum_{v=0}^{|\gamma_i|} \int_{\dctr} \sum_{n\in \N^\done} \Biggl|F^{(|\beta^\iota|+|\theta^\iota|+v)}\Biggl(\sum_{\jone=1}^\done b_\jone^\eta \langle n_\jone+s_\jone \rangle_\jone + |\mu|_\Sigma\Biggr)\Biggr|^2 \, \tilde\chi(|\eta|/M)  \,d\eta.
\end{multline*}
Let $S$ denote the unit sphere in $\ctr^*$. By passing to polar coordinates and exploiting the homogeneity of the $b_\jone^\eta$, the integral in the above formula is majorized by
\begin{multline*}
C \int_{S} \int_0^\infty \sum_{n\in \N^\done} \Biggl|F^{(|\beta^\iota|+|\theta^\iota|+v)}\Biggl(\rho \sum_{\jone=1}^\done b_\jone^\omega \langle n_\jone+s_\jone \rangle_\jone + |\mu|_\Sigma\Biggr)\Biggr|^2 \, \tilde\chi(\rho/M)  \rho^{\dtwo} \,\frac{d\rho}{\rho} \,d\omega \\
\leq C M^\dtwo \int_0^\infty |F^{(|\beta^\iota|+|\theta^\iota|+v)}(\rho + |\mu|_\Sigma)|^2 \int_S \sum_{n \in \N^\done} \tilde \chi(\rho/(M \langle n \rangle_{\omega,s})) \,d\omega \,\frac{d\rho}{\rho}
\end{multline*}
where  $\langle n \rangle_{\omega,s} = \sum_{\jone=1}^\done b_\jone^\omega \langle n_\jone+s_\jone \rangle_\jone \sim 1+|n|$ uniformly in $\omega \in S$ and $s \in J_\iota$.
Hence the sum in the right-hand side has at most $C_\iota (\rho/M)^\done$ summands, and the integral on $S$ is majorized by $C_\iota (\rho/M)^\done$. In conclusion, the left-hand side of \eqref{eq:weightedsummand2} is majorized by
\begin{multline*}
C_{K,\chi,\alpha} M^{\dtwo -2|\alpha|} \sum_{v=0}^{|\gamma_i|} \int_0^\infty |F^{(|\beta^\iota|+|\theta^\iota|+v)}(\rho + |\mu|_\Sigma)|^2 \, \rho^{\done-1} \,d\rho \\
\leq C_{K,\chi,\alpha} M^{\dtwo-2|\alpha|} \|F\|_{W_2^{|\alpha|}}^2,
\end{multline*}
because $d_1 \geq 1$, $\supp F \subseteq K$ and $|\beta^\iota| + |\theta^\iota| + |\gamma^\iota| \leq |\alpha|$, and we are done.
\end{proof}

\begin{prp}\label{prp:weightedl2}
Let $F : \R \to \C$ be smooth and such that $\supp F \subseteq K$ for some compact set $K \subseteq \R$. For all $r \in \leftclosedint 0, \dtwo/2 \rightopenint$,
\[
\int_{G} \left| (1+|u|)^r \, \Kern_{F(L)}(z,u) \right|^2 \,dz \,du \leq C_{K,r} \|F\|_{W_2^r}^2.
\]
\end{prp}
\begin{proof}
Take $\chi \in C^\infty_c(\leftopenint 0,\infty \rightopenint)$ such that $\supp \chi \subseteq \leftclosedint 1/2,2 \rightclosedint$ and $\sum_{k \in \Z} \chi(2^{-k} t) = 1$ for all $t \in \leftopenint 0,\infty \rightopenint$. If $F_M$ is defined for all $M \in \leftopenint 0,\infty \rightopenint$ as in Lemma~\ref{lem:weightedplancherel}, then $\Kern_{F_M(L,\vecU)}$ is given by the right-hand side of \eqref{eq:kernel}, where $m$ is defined by \eqref{eq:simplermultiplier}, and moreover
\[
\sum_{\jone=1}^\done b^\eta_\jone \langle n_\jone \rangle_\jone + |\mu|_\Sigma \geq C^{-1} |\eta|
\]
for all $\eta \in \dctr$, $\mu \in \leftclosedint 0,\infty \rightopenint^\tdone$ and $n \in \N^\done$, therefore $F_M(L,\vecU) = 0$ whenever $M > 2C\max K$.
Hence, if $k_K \in \Z$ is sufficiently large so that $2^{k_K} > 2 C \max K$, then
\[F(L) = \sum_{k \in \Z , \, k \leq k_K} F_{2^{k}}(L,\vecU)\]
(with convergence in the strong sense). Consequently an estimate for $\Kern_{F(L)}$ can be obtained, via Minkowski's inequality, by summing the corresponding estimates for $\Kern_{F_{2^{k}}}(L,\vecU)$ given by Lemma~\ref{lem:weightedplancherel}. If $r < d/2$, then the series $\sum_{k \leq k_K} (2^{k})^{\dtwo/2-r}$ converges, thus
\[
\int_{G} \left| |u|^r \, \Kern_{F(L)}(z,u) \right|^2 \,dz \,du \leq C_{K,r} \|F\|_{W_2^r}^2.
\]
The conclusion follows by combining the last inequality with the corresponding one for $r = 0$.
\end{proof}

Let $|\cdot|_\delta$ be a $\delta_t$-homogeneous norm on $G$; take, e.g., $|(z,u)|_\delta = |z| + |u|^{1/2}$. Interpolation then allows us to improve the standard weighted estimate for a homogeneous sublaplacian on a stratified group.

\begin{prp}\label{prp:improvedl2estimate}
Let $F : \R \to \C$ be smooth and such that $\supp F \subseteq K$ for some compact set $K \subseteq \R$. For all $r \in \leftclosedint 0, \dtwo/2 \rightopenint$, $\alpha \geq 0$ and $\beta > \alpha + r$,
\begin{equation}\label{eq:improvedl2estimate}
\int_{G} \left| (1+|(z,u)|_\delta)^\alpha \, (1+|u|)^r \, \Kern_{F(L)}(z,u) \right|^2 \,dz \,du \leq C_{K,\alpha,\beta,r} \|F\|_{W_2^\beta}^2.
\end{equation}
\end{prp}
\begin{proof}
Note that $1+|u| \leq C (1+|(z,u)|_\delta)^2$. Hence, in the case $\alpha \geq 0$, $\beta > \alpha + 2r$, the inequality \eqref{eq:improvedl2estimate} follows by the mentioned standard estimate (see \cite[Lemma 1.2]{mauceri_vectorvalued_1990} or \cite[Theorem~2.7]{martini_joint_2012}). On the other hand, if $\alpha = 0$ and $\beta \geq r$, then \eqref{eq:improvedl2estimate} is given by Proposition~\ref{prp:weightedl2}. The full range of $\alpha$ and $\beta$ is then obtained by interpolation.
\end{proof}

We can finally prove the crucial estimate.

\begin{proof}[Proof of Proposition~\ref{prp:l1estimate}]
Take $r \in \leftopenint (\dim G)/2 + \dtwo/2-s,\dtwo/2 \rightopenint$. Then
\[s-r > (\dim G)/2 + \dtwo/2 - 2r = (\dim \fst)/2 + \dtwo-2r,\]
hence we can find $\alpha_1 > (\dim \fst)/2$ and $\alpha_2 > \dtwo-2r$ such that $s - r > \alpha_1 + \alpha_2$. Set $w_s(z,u) = (1+|(z,u)|_\delta)^\alpha \, (1+|u|)^r$. The $L^2$ estimate \eqref{eq:thel2estimate} then follows from Proposition~\ref{prp:improvedl2estimate}. On the other hand, for all $(z,u) \in G$,
\[
w_s^{-2}(z,u) \leq C_s (1+|z|)^{-2\alpha_1} \, (1+|u|)^{-\alpha_2 - 2r},
\]
and the right-hand side is integrable over $G \cong \lie{v} \times \lie{z}$ since $2\alpha_1 > \dim\fst$ and $\alpha_2 + 2r > \dtwo = \dim\ctr$. Therefore $w_s^{-1} \in L^2(G)$, and The $L^1$ estimate \eqref{eq:thel1estimate} follows from \eqref{eq:thel2estimate} and H\"older's inequality.
\end{proof}

\section{Remarks on the validity of the assumption and direct products}

In this section we do no longer suppose that $G$ and $L$ are a $2$-step stratified Lie group and a sublaplacian satisfying Assumption \AssA.

As observed in \S\ref{section:assumption}, a necessary condition for the validity of Assumption \AssA\  is that the skewadjoint endomorphism $J_\eta$ of the first layer $\lie{v}$ has constant rank for $\eta$ ranging in $\dctr = \ctr^* \setminus \{0\}$. Here we show that this condition is also sufficient when the rank is minimal.

\begin{prp}
Let $G$ be a $2$-step nilpotent Lie group, with Lie algebra $\lie{g} = \lie{v} \oplus \lie{z}$, and let $\langle \cdot,\cdot \rangle$ be an inner product on $\lie{v}$. Suppose that the skewadjoint endomorphism $J_\eta$ of $\lie{v}$ has rank $2$ for all $\eta \in \dctr$. Then $G$ satisfies Assumption \AssA\  with the sublaplacian $L$ associated to the given inner product, and also with any other sublaplacian associated to an inner product on a complement of $\lie{z}$.

Let moreover $G_\C$ be the complexification of $G$, considered as a real $2$-step group, with Lie algebra $\lie{g}_\C = \lie{v}_\C \oplus \lie{z}_\C$, and let $\lie{v}_\C$ be endowed with the real inner product induced by the inner product on $\lie{v}$. Then $G_\C$, with the sublaplacian associated to the given inner product, satisfies Assumption \AssA.
\end{prp}
\begin{proof}
From the normal form for skewadjoint endomorphisms, it follows immediately that, if $J_\eta$ has rank $2$, then $J_\eta^2$ has exactly one nonzero eigenvalue, and Assumption \AssA\  is trivially verified. Moreover, if $\lie{v}$ is identified with $\lie{g}/\lie{z}$, then $\ker J_\eta$ corresponds to the subspace
\[
N_\eta = \{ x + \lie{z} \tc x \in \lie{g} \text{ and } \eta([x,x']) = 0 \text{ for all } x' \in \lie{g}\}
\]
of $\lie{g}/\lie{z}$; hence the rank condition on $J_\eta$ can be rephrased by saying that $N_\eta$ has codimension $2$ for all $\eta \in \dctr$, and this condition does not depend on the sublaplacian $L$ chosen on $G$.

Let $R(J_\eta)$ denote the range of $J_\eta$. We show now that, for all $\eta,\eta' \in \dctr$, the intersection $R(J_\eta) \cap R(J_{\eta'})$ is nontrivial. If it were trivial, since $J_{\eta+\eta'} = J_\eta + J_\eta'$, we would have $\ker J_{\eta+\eta'} = \ker J_\eta \cap \ker J_{\eta'}$, hence
\[R(J_{\eta+\eta'}) = (\ker J_{\eta+\eta'})^\perp = R(J_\eta) \oplus R(J_{\eta'}),\]
thus $J_{\eta+\eta'}$ would have rank $4$, contradiction.

Consider now the complexification $\lie{g}_\C = \lie{g} \oplus i \lie{g}$. Via the linear identifications $\lie{g}_\C = \lie{g} \times \lie{g}$, $\lie{z}^*_\C = \lie{z}^* \times \lie{z}^*$, $\lie{v}_\C = \lie{v} \times \lie{v}$, the skewsymmetric endomorphism $\tilde J_\eta$ of the first layer $\lie{v}_\C$ corresponding to the element $\eta =(\eta_R,\eta_I) \in \lie{z}_\C^*$ is given by
\begin{equation}\label{eq:complexification}
\tilde J_\eta (x_R,x_I) = (J_{\eta_R} x_R + J_{\eta_I} x_I, J_{\eta I} x_R - J_{\eta_R} x_I).
\end{equation}

Take now $\eta = (\eta_R,\eta_I) \in \dctr_\C$; we want to show that $\tilde J_\eta^2$ has rank $4$ and a unique nonzero eigenvalue. We distinguish several cases.

If $\eta_I = 0$, then $\tilde J_\eta = J_{\eta_R} \times (-J_{\eta_R})$, hence $\tilde J_\eta^2 = J_{\eta_R}^2 \times J_{\eta_R}^2$ satisfies the condition. The same argument gives the conclusion in the case $\eta_R = 0$.

If both $\eta_R,\eta_I \in \dctr$, then $R(J_{\eta_R}) \cap R(J_{\eta_I}) \neq 0$, hence $\dim(R(J_{\eta_R}) \cap R(J_{\eta_I}))$ is either $2$ or $1$. In the first case, $R(J_{\eta_R}) = R(J_{\eta_I})$, so $J_{\eta_R}$ and $J_{\eta_I}$ commute and \eqref{eq:complexification} implies that
\[
\tilde J_\eta^2 = (J_{\eta_R}^2 + J_{\eta_I}^2) \times (J_{\eta_R}^2 + J_{\eta_I}^2);
\]
since $J_{\eta_R}^2$ and $J_{\eta_I}^2$ are negative multiples of the same orthogonal projection, the conclusion follows.

Suppose now that $R(J_{\eta_R}) \cap R(J_{\eta_I}) = \R x$ for some unit vector $x \in \lie{v}$, and set $y_R = J_{\eta_R}x$, $y_I = J_{\eta_I}x$, $b_R = |y_R|$, $b_I = |y_I|$; in particular $J_{\eta_R}^2 x = -b_R^2 x$ and $J_{\eta_I}^2 x = -b_I^2 x$. Since $J_{\eta_R}$ and $J_{\eta_I}$ are skewadjoint and of rank $2$, necessarily $J_{\eta_R} x, J_{\eta_I}x \in x^\perp$ and $J_{\eta_R}(x^\perp) = J_{\eta_I}(x^\perp) = \R x$, therefore $J_{\eta_R} J_{\eta_I} x$ and $J_{\eta_I} J_{\eta_R} x$ are both multiples of $x$; on the other hand,
\[
\langle J_{\eta_R} J_{\eta_I} x, x \rangle = - \langle J_{\eta_I} x, J_{\eta_R} x \rangle = \langle x, J_{\eta_I} J_{\eta_R} x \rangle,
\]
hence $J_{\eta_R} J_{\eta_I} x = J_{\eta_I} J_{\eta_R} x$. This identity, together with \eqref{eq:complexification}, allows us easily to show that
\begin{gather*}
\tilde J_\eta(x,0) = (y_R,y_I), \qquad \tilde J_\eta(y_R,y_I) = -(b_R^2 + b_I^2) (x, 0),\\
\tilde J_\eta(0,x) = (y_I,-y_R), \qquad \tilde J_\eta(y_I,-y_R) = -(b_R^2 + b_I^2) (0, x).
\end{gather*}
Note that $b_R^2 + b_I^2$ is the squared norm of both $(y_R,y_I)$ and $(y_I,-y_R)$. Hence we would be done if we knew that $R(\tilde J_\mu)$ coincides with the linear span $W$ of $(x,0)$, $(0,x)$, $(y_R,y_I)$, $(y_I,-y_R)$.

In fact, we just need to show that $R(\tilde J_\eta)$ is contained in $W$, or equivalently, that $W^\perp$ is contained in $\ker \tilde J_\eta$. On the other hand, if $v = (v_R,v_I) \in W^\perp$, then $v_R,v_I \in x^\perp$ and moreover
\[\langle v_R, y_R \rangle + \langle v_I, y_I \rangle = 0, \qquad \langle v_R, y_I \rangle - \langle v_I, y_R \rangle = 0,\]
hence $J_{\eta_R} v_R, J_{\eta_R} v_I, J_{\eta_I} v_R, J_{\eta_I} v_I \in \R x$, and
\[\langle J_{\eta_R} v_R, x \rangle = - \langle v_R, y_R \rangle = \langle v_I, y_I \rangle = -\langle J_{\eta_I} v_I, x \rangle,\]
\[\langle J_{\eta_I} v_R, x \rangle = - \langle v_R, y_I \rangle = - \langle v_I, y_R \rangle = \langle J_{\eta_R} v_I, x \rangle,\]
therefore $J_{\eta_R} v_R = -J_{\eta_I} v_I$ and $J_{\eta_I} v_R = J_{\eta_R} v_I$, from which it follows immediately that $\tilde J_\eta(v_R,v_I) = 0$.
\end{proof}

The next proposition shows how groups and sublaplacians satisfying Assumption \AssA\ may be ``glued together'', so to give a higher-dimensional group and a sublaplacian that satisfy Assumption \AssA\  too.

\begin{prp}
Suppose that, for $j=1,2$, the sublaplacian $L_j$ on the $2$-step stratified Lie group $G_j$ satisfies Assumption \AssA. Suppose further that the centers of $G_1$ and $G_2$ have the same dimension. Let $G$ be the quotient of $G_1 \times G_2$ given by any linear identification of the respective centers, and let $L = L_1^\sharp + L_2^\sharp$, where $L_j^\sharp$ is the pushforward of $L_j$ to $G$. Then the sublaplacian $L$ on the group $G$ satisfies Assumption \AssA.
\end{prp}
\begin{proof}
Let $\lie{g}_j$ be the Lie algebra of $G_j$, and let $\fst_j$ and $\langle \cdot, \cdot \rangle_j$ be the linear complement of the center $\ctr_j$  and the inner product on $\fst_j$ determined by the sublaplacian $L_j$; denote moreover by $J_{j,\eta}$ the skewadjoint endomorphism of $\fst_j$ determined by $\eta \in \ctr_j^*$.

The linear identification of the centers of $G_1$ and $G_2$ corresponds to a linear isomorphism $\phi : \lie{z}_1 \to \lie{z}_2$, and the Lie algebra $\lie{g}$ of the quotient $G$ can be identified with $\lie{v}_1 \times \lie{v}_2 \times \lie{z}_2$, with Lie bracket
\[
[(v_1,v_2,z),(v_1',v_2',z')] = (0,0,\phi([v_1,v_1']) + [v_2,v_2']).
\]
Then the sublaplacian $L$ on $G$ corresponds to the inner product $\langle \cdot, \cdot \rangle$ on $\lie{v}_1 \times \lie{v}_2$ defined by
\[
\langle (v_1,v_2) , (v_1',v_2') \rangle = \langle v_1, v_1' \rangle_1 + \langle v_2, v_2' \rangle_2.
\]
In particular, if $\phi^* : \lie{z}_2^* \to \lie{z}_1^*$ denotes the adjoint map of $\phi : \lie{z}_1 \to \lie{z}_2$, then it is easily checked that the skewadjoint endomorphism of the first layer $\lie{v}_1 \times \lie{v}_2$ of $\lie{g}$ corresponding to an element $\eta$ of the dual $\lie{z}_2^*$ of the center of $\lie{g}$ is given by $J_\eta = J_{1,\phi^* \eta} \times J_{2,\eta}$. Hence the orthogonal decomposition of $\lie{v}_1 \times \lie{v}_2$ giving the ``simultaneous diagonalization'' of the $J_\eta$ for all $\eta \in \dctr_2$ (in the sense of \S\ref{section:assumption}) is simply obtained by juxtaposing the corresponding orthogonal decompositions of $\lie{v}_1$ and $\lie{v}_2$.
\end{proof}

Note that the direct product $G_1 \times G_2$ itself need not satisfy Assumption \AssA, even if the factors $G_1$ and $G_2$ do. However a functional-analytic argument, as in \cite[\S 4]{mller_spectral_1994}, can be used to deal with that case.

The key step in our proof of Theorem~\ref{thm:theorem} is the weighted $L^2$ estimate \eqref{eq:thel2estimate} of Proposition~\ref{prp:l1estimate}. Let us now turn the conclusion of Proposition~\ref{prp:l1estimate} into an assumption on a homogeneous sublaplacian $L$ on a stratified group $G$.

\begin{assB}
For all $s > t$ there exist a weight $w_s : G \to \leftclosedint 1,\infty \rightopenint$ such that $w_s^{-1} \in L^2(G)$ and, for all compact sets $K \subseteq \R$ and all Borel functions $F : \R \to \C$ with $\supp F \subseteq K$,
\begin{equation}\label{eq:assl2estimate}
\|w_s \, \Kern_{F(L)} \|_{L^2(G)} \leq C_{K,s} \|F\|_{W_2^s(\R)}.
\end{equation}
\end{assB}

Our Proposition~\ref{prp:l1estimate} can then be rephrased by saying that Assumption \AssA\ implies Assumption \AssB{t} for $t = (\dim G)/2$. Note, on the other hand, that Assumption \AssB{t} makes sense for homogeneous sublaplacians on stratified groups $G$ of step other than $2$. In fact, every homogeneous sublaplacian on a stratified group of homogeneous dimension $Q$ satisfies Assumption \AssB{t} for $t = Q/2$, by \cite[Lemma~1.2]{mauceri_vectorvalued_1990} (suitably extended so to admit multipliers that do not vanish in a neighborhood of the origin of $\R$; see, e.g., \cite[Lemma~3.1]{mller_spectral_1994} for the $1$-dimensional case, and \cite[Theorem~2.7]{martini_joint_2012} for the higher-dimensional case).

Differently from Assumption \AssA, the new Assumption \AssB{t} ``behaves well'' under direct products.

\begin{prp}\label{prp:product}
For $j=1,\dots,n$, let $L_j$ be a homogeneous sublaplacian on a stratified Lie group $G_j$ satisfying Assumption \AssB{t_j} for some $t_j > 0$. Let $G = G_1 \times \dots \times G_n$ and $L = L_1^\sharp + \dots + L_n^\sharp$, where $L_j^\sharp$ is the pushforward to $G$ of the operator $L_j$. Then the sublaplacian $L$ on $G$ satisfies Assumption \AssB{t}, where $t = t_1+\dots+t_n$.
\end{prp}
\begin{proof}
Take $s > t$. Then we can choose $s_1,\dots,s_n$ such that $s_1 > t_1,\dots,s_n> t_n$ and $s = s_1+\dots+s_n$. Let then $w_{j,s_j} : G_j \to \leftclosedint 1,\infty \rightopenint$ be the weight corresponding to $s_j$ given by Assumption \AssB{t} on $G_j$ and $L_j$, for $j=1,\dots,n$. In particular $w_{j,s_j}^{-1} \in L^2(G_j)$ and, for all $\phi \in C^\infty_c(\R)$, the map $F \mapsto \Kern_{(\phi F)(L_j)}$ is a bounded linear map of Hilbert spaces $W_2^{s_j}(\R) \to L^2(G_j,w_{j,s_j}^2(x_j) \,dx_j)$, where $dx_j$ denotes the Haar measure on $G_j$.

The operators $L_1^\sharp,\dots,L_n^\sharp$ on $G$ are essentially self-adjoint and commute strongly, that is, they admit a joint spectral resolution and a joint functional calculus on $L^2(G)$, and moreover, for all bounded Borel functions $F_1,\dots,F_n : \R \to \C$,
\[
\Kern_{(F_1 \otimes \dots \otimes F_n)(L_1^\sharp,\dots,L_n^\sharp)} = \Kern_{F_1(L_1)} \otimes \dots \otimes \Kern_{F_n(L_n)}
\]
\cite[Corollary~5.5]{martini_spectral_2011}. Hence, for all $\phi_1,\dots,\phi_n \in C^\infty_c(\R)$, if $\phi = \phi_1 \otimes \dots \otimes \phi_n$, then the map $H \mapsto \Kern_{(\phi H)(L_1^\sharp,\dots,L_n^\sharp)}$ is the tensor product of the maps $F_j \mapsto \Kern_{(\phi_j F_j)(L_j)}$. Since these maps are bounded $W_2^{s_j}(\R) \to L^2(G_j,w_{j,s_j}^2(x_j) \,dx_j)$, the map $H \mapsto \Kern_{(\phi H)(L_1^\sharp,\dots,F_n^\sharp)}$ is bounded $S_2^{(s_1,\dots,s_n)}W(\R^n) \to L^2(G,w_s^2(x) \,dx)$, where $S^{(s_1,\dots,s_n)}_2 W(\R^n) = W_2^{s_1}(\R) \otimes \dots \otimes W_2^{s_n}(\R)$ is the $L^2$ Sobolev space with dominating mixed smoothness \cite{schmeisser_recent_2007} of order $(s_1,\dots,s_n)$, and $w_s = w_{1,s_1} \otimes \dots \otimes w_{n,s_n}$ is the product weight on $G$. In particular, for all compact sets $K \subseteq \R$, if we choose the cutoffs $\phi_j \in C^\infty_c(\R)$ so that $\phi_j|_K = 1$, then we deduce that, for all $H : \R^n \to \C$ with $\supp H \subseteq K^n$,
\[
\|w_s \, \Kern_{H(L_1^\sharp,\dots,L_n^\sharp)} \|_{L^2(G)} \leq C_{K,s} \|H\|_{S^{(s_1,\dots,s_n)}_2 W(\R^n)}.
\]
(cf.\ \cite[Proposition~5.2]{martini_joint_2012}). Since
\[\begin{split}
\|f\|_{S^{(s_1,\dots,s_n)}_2 W(\R^n)}^2 &\sim \int_{\R^n} |\hat f(\xi)|^2 (1+|\xi_1|)^{2s_1} \dots (1+|\xi_n|)^{2s_n} \,d\xi \\
&\leq \int_{\R^n} |\hat f(\xi)|^2 (1+|\xi|)^{2s_1+\dots+2s_n} \,d\xi \sim \|f\|^2_{W_2^{s}(\R^n)},
\end{split}\]
where $\hat f$ denotes the Euclidean Fourier transform of $f$, we see immediately that the estimate
\begin{equation}\label{eq:productestimate}
\|w_s \, \Kern_{H(L_1^\sharp,\dots,L_n^\sharp)} \|_{L^2(G)} \leq C_{K,s_1,\dots,s_n} \|H\|_{W_2^{s}(\R^n)},
\end{equation}
holds true whenever $K \subseteq \R$ is compact and $H : \R^n \to \C$ is supported in $K^n$.

Take now a compact set $K \subseteq \R$ and choose a smooth cutoff $\eta_K \in C^\infty_c(\R)$ such that $\eta_K|_{\leftclosedint 0,\max K\rightclosedint} = 1$. Let $F : \R \to \C$ be such that $\supp F \subseteq K$, and define $H : \R^n \to \C$ by
\[
H(\lambda_1,\dots,\lambda_n) = F(\lambda_1 + \dots + \lambda_n) \, \eta_K(\lambda_1) \dots \, \eta_K(\lambda_n)
\]
for all $(\lambda_1,\dots,\lambda_n) \in \R^n$. Then $\supp H \subseteq (\supp \eta_K)^n$, and
\[
F(\lambda_1 + \dots + \lambda_n) = H(\lambda_1,\dots,\lambda_n)
\]
for all $(\lambda_1,\dots,\lambda_n) \in \leftclosedint 0,\infty \rightopenint^n$. Since the operators $L_1,\dots,L_n$ are nonnegative, the joint spectrum of $L_1^\sharp,\dots,L_n^\sharp$ is contained in $\leftclosedint 0,\infty \rightopenint^n$, hence
\[
F(L) = F(L_1^\sharp + \dots + L_n^\sharp) = H(L_1^\sharp,\dots,L_n^\sharp).
\]
Consequently, by \eqref{eq:productestimate} and the smoothness of the map $(\lambda_1,\dots,\lambda_n) \mapsto \lambda_1+\dots+\lambda_n$ we obtain that
\[
\|w_s \Kern_{F(L)} \|_{L^2(G)} \\
\leq C_{K,s} \|H\|_{W_2^s(\R^n)} \leq C_{K,s} \|F\|_{W_2^s(\R)}.
\]
Since clearly $w_s^{-1} = w_{1,s_1}^{-1} \otimes \dots \otimes w_{n,s_n}^{-1} \in L^2(G)$, we are done.
\end{proof}

The previous results, together with the known weighted estimates for abelian \cite[Lemma~3.1]{mller_spectral_1994} and M\'etivier \cite{hebisch_multiplier_1993,hebisch_multiplier_1995,martini_joint_2012} groups, then yield the following extension of Theorem~\ref{thm:theorem}.

\begin{thm}
For $j=1,\dots,n$, suppose that $L_j$ is a homogeneous sublaplacian on a stratified Lie group $G_j$. Suppose further that, for each $j \in \{1,\dots,n\}$, at least one of the following conditions holds:
\begin{itemize}
\item $G_j$ and $L_j$ satisfy Assumption \AssA;
\item $G_j$ is a M\'etivier group;
\item $G_j$ is abelian.
\end{itemize}
Let $G = G_1 \times \dots \times G_n$ and $L = L_1^\sharp + \dots + L_n^\sharp$, as in Proposition~\ref{prp:product}. If $F : \R \to \C$ satisfies
\[\|F\|_{MW_2^s} < \infty\]
for some $s > (\dim G)/2$, then $F(L)$ is of weak type $(1,1)$ and bounded on $L^p(G)$ for all $p \in \leftopenint 1,\infty \rightopenint$.
\end{thm}

\section*{Acknowledgments}

I am deeply grateful to Detlef M\"uller for stimulating discussions and his continuous encouragement. I am also grateful for the opportunity of visiting the Mathematical Institute of the University of Wroc\l aw in the first bimester of 2011; numerous discussions, particularly with Jacek Dziuba\'nski and Waldemar Hebisch, have greatly contributed to my understanding of problems related to spectral multipliers.

\bibliographystyle{amsabbrv}
\bibliography{heisenbergreiter}

\end{document}